 \documentclass[12pt]{amsart}

\theoremstyle{plain}
\newtheorem{thm}{Theorem}

\newtheorem{lemma}[thm]{Lemma}
\newtheorem{corollary}[thm]{Corollary}

\newcommand{\cid}{\stackrel{d}{\longrightarrow}}

\newcommand{\toi}{\to\infty}

\newcommand{\N}{\mathbb{N}}

\begin{document}


\title{Extremes of random variables observed in renewal times}
\maketitle

\begin{center}
\author{Bojan Basrak}\\
\address{Department of Mathematics, University of Zagreb, Croatia}\\
\email{bbasrak@math.hr}
\end{center}

\vspace{0.5cm}
\begin{center}
\author{Drago \v Spoljari\'c}\\
\address{Faculty of Mining, Geology and Petroleum Engineering, University of Zagreb, Croatia}\\
\email{drago.spoljaric@rgn.hr}
\end{center}



\begin{abstract}
We use point processes theory to describe the asymptotic distribution of all upper order statistics for observations collected at renewal times. As a corollary, we obtain limiting theorems for corresponding extremal processes.
\end{abstract}






\section{Introduction}
\label{sec_Intro}
The maximum of a random number of random variables has been studied
for decades. The basic problem is to understand 
the distribution of
\begin{equation*}
M(t) = \max_{i=1,\ldots, \tau(t)} X_i\,,
\end{equation*}
for  an iid sequence $(X_n)$ and  random variables $\tau(t),\, t\geq 0,$ which are typically modeled by a renewal process.  The earliest references are \cite{Lamperti61,Berman62,BarndorffNielsen64}, see also  \cite{SS83,Anderson87} for
extensions and applications in engineering.
More recently, \cite{MStoev2009} and \cite{PMM2009} 
studied the convergence of the process $(M(t))$ towards an appropriate extremal process.
A general treatment of extremal processes with random sample size
 can be found in \cite{SilvestrovTeugels98,SilvestrovTeugels2004}.

From practical perspective, it is frequently important to understand the distribution of all 
the extreme observations and not merely the maximum.
Thus, we aim to explain the
  limiting behaviour of all large values in the sequence $(X_n)$, which
arrive before a given time $t$. To do that, we rely on the theory of point processes.
Such an approach seems to be entirely new in this context. 
It does  not only yield more general results, but
we believe, it provides a better insight into why  previously established results actually hold.

Throughout $(\tau(t))$ represents the renewal process generated by an iid sequence of nonnegative random variables $(Y_n)$, i.e. 
\begin{equation}\label{def_renewal_process}
\tau(t) = \inf \{ k: Y_1+\cdots + Y_k >t \}\,,
\quad \mbox{ for } t\geq 0\,.
\end{equation}
Moreover, we assume that the distribution of $X_1$ belongs to 
the maximum domain of attraction (MDA for short) of one of the three extreme value distributions, denoted by $G$.
Because of the correspondence between MDA's of Fr\'echet and Weibull distributions, we discuss only Gumbel and Fr\'echet MDA's in detail (see subsection 3.3.2 in \cite{EKMikosch}).
Recall that $X_1\in\mathrm{MDA}(G)$ means that for some sequences $(a_n)$ and ($b_n$)
$$
 n P( X_1 > a_n x + b_n) \to -\log G(x)\,, 
$$ 
as $n \to \infty$, for each real $x$ such that $G(x)\in(0,1)$. This is further equivalent to
$
 {(M_n - b_n)}/{a_n}  \cid G\,,  
$
as $n\toi$, where  $M_n = \max\{ X_i: 1\leq i \leq n \}$
denotes the partial maxima of the iid sequence $(X_n)$.
Typically, one also assumes that $(\tau(t))$ is  independent from sequence of observations $(X_n)$.

The partial maximum of $(X_n)$ governed by $\tau(t)$ is defined as
\begin{equation*}
M^\tau(t) = \sup \{ X_i: i \leq \tau(t) \}\,.
\end{equation*}
If the  steps of the renewal process have finite mean,
that is, if $\mu=EY\in (0, \infty)$, for iid $X_n$'s,
 we know that
 the partial maxima governed by the renewal process behave as if
 they were observed at deterministic times. In other words
\begin{equation} \label{eq:McidG}
\frac{M^\tau (t) - b_{\lfloor t/\mu \rfloor} }{ a_{\lfloor t/\mu \rfloor}}  \cid  G \,,
\end{equation}
as $t\toi$. 
Intuitively, one could say that $M^\tau (t)$ behaves as $M_{\lfloor t/\mu \rfloor}$. Moreover, this holds
irrespective of dependence between $(\tau(t))$ and the observations.
For the renewal process with infinite mean, but with regularly varying steps, one can still determine
 the
asymptotic distribution of the maximum, see \cite{Berman62}.
In such a setting, the convergence of $(M^\tau(t))$ was shown at the level of stochastic
processes, see \cite{MStoev2009,PMM2009}. 
In the rest of the paper we show how one can 
move beyond the maxima  and
 extend those results to all upper order statistic in both finite and infinite mean case.

The paper is organised as follows:  notation and auxiliary results are introduced in section~\ref{section_prelims}. In section~\ref{section_finite_mean}, we consider the finite mean case in detail, while the problem when interarrival times have infinite mean and are independent of observations will be studied in section 4.

\section{Auxiliary point processes}\label{section_prelims}

As already mentioned in the introduction, we assume that the distribution of $X_1$ belongs to the MDA(G) where $G$ is Gumbel $(G=\Lambda)$ or Fr\'echet $(G=\Phi_\beta$, for $\beta>0)$ distribution. In particular, there exist functions $a(t)$ and $b(t)$ such that 
\begin{equation}\label{MDA_condition}
 t P( X_1 > a(t) x + b(t)) \to -\log G(x)\,, 
\end{equation}
as $t\toi$ (cf. \cite{ResEVRVPP}).
Throughout the article  we consider  point processes of the form
\begin{equation}\label{def_N_t_general}
  N_t = \sum_{i\geq 1} \delta_{\left({i}/{g(t)}, X_{t,i}\right)}\,,
\end{equation}
for a nondecreasing function $g:(0,\infty)\to (0,\infty)$ tending to $+\infty$ as $x
\toi$, with
\begin{equation}\label{def_Xti}
X_{t,i} = \frac{X_i - b(g(t))}{a(g(t))}\,,
\end{equation}
where scaling and centering functions $a(t)$ and $b(t)$ are given in $(\ref{MDA_condition})$.

In the sequel we will allow the function $g$ to depend on the tail of the step size distribution.
However, for  iid observations $(X_n)$, it is well known that
$X_1\in\mathrm{MDA}(G)$ is both necessary and sufficient for weak convergence of 
$(N_t)$.
Moreover,  the limiting point process, $N$  say, is a Poisson random measure (PRM) 
with mean measure $\lambda \times \mu_G$ ($\mathrm{PRM}(\lambda\times \mu_G)$ for short), where $\lambda$ denotes the Lebesgue measure and $\mu_G$ represents the measure induced by the nondecreasing function $\log G$.  Observe that $N_t$ take value in the space of Radon point measures
 $M_p([0,\infty)\times \mathbb{E})$, with $\mathbb{E}$ depending on $G$. For instance, in the Gumbel MDA, with $G=\Lambda$,  $\mathbb{E}=(-\infty,\infty]$ and the measure $\mu_G$ 
 satisfies $\mu_G(x,\infty]=e^{-x}$ for $x\in\mathbb{R}$. In the Fr\'echet MDA, with $G=\Phi_\alpha$,  $\mathbb{E}=(0,\infty]$ and  the measure $\mu_G$ 
 satisfies $\mu_G(x,+\infty]=x^{-\alpha}$ for every $x>0$. For
 the Weibull case and the definition of vague topology on the space of  point measures
 $M_p([0,\infty)\times \mathbb{E})$ we refer to~\cite{ResEVRVPP}. 

Since the distribution of point processes $N_t$ contains the information about all upper order statistics in the sequence $(X_n)$,
our plan is
 to show the convergence of point processes $N_t$  restricted to time intervals determined by a renewal process.  An  application of the continuous mapping theorem together with 
  Proposition 3.13 in \cite{ResEVRVPP} yields 
the following simple result, which plays an important role in 
the sequel.

\begin{lemma}\label{basiclemma}
Assume that $N,(N_t)_{t \geq 0}$ are point processes with values in $M_p([0,\infty)\times\mathbb{E})$, for a measurable subset $\mathbb{E}$ of $ \mathbb{R}^d$. Assume further that $Z,(Z_t)_{t \geq 0}$  are $\mathbb{R}_{+}$-valued random variables. If $P(N(\{Z\}\times \mathbb{E})>0) = 0$ and
\begin{equation}\label{joint_cvg}
(N_t,Z_t)\stackrel{d}{\longrightarrow}(N,Z)\,,
\end{equation}
as $t\toi$, 
then
\begin{equation*}
N_t\Big|_{[0,Z_t]\times\mathbb{E}} \cid N\Big|_{[0,Z]\times\mathbb{E}}\,,
\end{equation*}
as $t\toi$.
\end{lemma}
By $m\vert_A$ above we denote the restriction of a point 
measure $m$ on a set $A$, i.e. $m\vert_A(B) = m(A\cap B)$.

Clearly, the joint convergence in (\ref{joint_cvg}) follows at once from $N_t\stackrel{d}{\rightarrow}N$ and $Z_t\stackrel{d}{\rightarrow}Z$ if the limit $Z$ is a constant or if $N_t$ and $Z_t$ are independent. 
In the sequel we consider $Z_t$ as a passage time of a renewal process, that is $Z_t$ will be determined by the suitably normalised random variable $\tau(t)$.
In such a setting, the
treatment of the joint convergence in (\ref{joint_cvg}) depends on the mean of interarrival times. In the finite mean case the convergence follows easily, while in the infinite mean case things get more complicated. However, if the  steps $(Y_n)$  are regularly varying with index $\alpha\in (0,1)$ (cf.~\cite{Anderson87,MS2004,PMM2009})
and independent of the observations, the limiting distribution of the upper order statistics can be
determined precisely as we show below.

\section{Observations governed by a finite mean renewal process}\label{section_finite_mean}

The finite mean case is well understood, still Theorem~\ref{thm_const} below
represents a generalization of the previously published results to the
point processes setting, which allows one  to describe the joint limiting distribution of all the upper order statistics. 

 Recall that $(\tau(t))$ is the renewal process generated by an iid sequence of nonnegative random variables $(Y_n)$. In this section we assume that  $\mu=EY_1 <\infty$. By the strong law of large numbers (SLLN)  it follows that for every $c\geq 0$,
\begin{equation}\label{eq:SLLNCP}
\frac{\mu\,\tau(tc)}{t}\stackrel{a.s.}{\longrightarrow}c\,,
\end{equation}
as $t\toi$ (see \cite{GutSRW}).  In this setting,
the normalizing function $g$ in the definition of point process $N_t$ (see (\ref{def_N_t_general})) is simply $g(t)=t/\mu$, that is  we set
\begin{equation*}
  N_t = \sum_{i\geq 1} \delta_{(\frac{i}{t/\mu}, X_{t,i})}\,,
\end{equation*}
with $X_{t,i}$ defined as in (\ref{def_Xti}).
Applying Lemma~\ref{basiclemma} to $N_t$ and $Z_t = \mu \tau(tc)/t$, one can show the following theorem which describes asymptotic behaviour of all the upper order statistics in the sequence $(X_n)$ until the passage time $\tau(t)$. In particular, 
the statement (\ref{eq:McidG}) is its immediate consequence.

\begin{thm}\label{thm_const}
Suppose $(X_n)$ is an iid sequence such that $X_1\in\mathrm{MDA}(G)$. If $\mu=EY_1<\infty$,
then, for every $c\geq 0$,
\begin{equation*}
N_t\Big|_{\left[0,\frac{\mu\,\tau(tc)}{t}\right]\times\mathbb{E}} \cid  N\Big|_{\left[0,c\right]\times\mathbb{E}}\,,
\end{equation*}
as $t\toi$, where $N$ is $\mathrm{PRM}(\lambda\times\mu_G)$ and  the 
measure $\mu_G$ and set $\mathbb{E}$ correspond to $G$ as described in section 2.
\end{thm}

\begin{proof}
The assumption $X_1\in\mathrm{MDA}(G)$ is equivalent to $N_t\stackrel{d}{\longrightarrow}N$ as $t\toi$, where $N$ is $\mathrm{PRM}(\lambda\times\mu_G)$. Due to the assumption $\mu<\infty$, using (\ref{eq:SLLNCP}) and Slutsky's theorem (see Theorem 3.9 in \cite{BillingsleyCPM}), one can conclude that
\begin{equation*}
\left(N_t,\frac{\mu\,\tau(tc)}{t}\right)\cid (N,c)\,, 
\end{equation*}
as $t\toi$, for every $c\geq 0$. The convergence takes place in $M_p([0,\infty)\times\mathbb{E})\times \mathbb{R}_+$ endowed with the product topology i.e. topology of vague convergence of point measures and standard topology on $\mathbb{R}_+$ generated by the open intervals. At the end, an application of Lemma~\ref{basiclemma} to $Z_t = \mu\tau(tc)/t$ yields the desired result.
\end{proof}
Note that Theorem~\ref{thm_const} makes no restriction on dependence between the sequences $(X_n)$ and $(Y_n)$. One consequence of Theorem~\ref{thm_const} is the joint limiting distribution of all upper order statistics. As an example we derive the joint distribution for the largest and the second largest observation. Let $M_k^{\tau}(t)$, for $k\in\N$  and $t > 0$, represents the $k$-th upper order statistics in a sample of observations $\{X_1\ldots,X_{\tau(t)}\}$. Clearly, for real $x_1>x_2$ and $c>0$, it holds
\begin{eqnarray*}
\lefteqn{P(M_{1}^\tau(tc)\leq a(t/\mu)x_1+b(t/\mu),M_{2}^\tau(tc)\leq a(t/\mu)x_2+b(t/\mu))}\\
&=&P\left(N_t\left(\left[0,\frac{\tau(tc)}{t/\mu}\right]\times(x_1,\infty]\right)=0,N_t\left(\left[0,\frac{\tau(tc)}{t/\mu}\right]\times(x_2,\infty]\right)\leq 1\right)\,.
\end{eqnarray*}
Therefore, by Theorem~\ref{thm_const}, as $t\toi$, we get
\begin{eqnarray*}
\lefteqn{P(M_{1}^\tau(tc)\leq a(t/\mu)x_1+b(t/\mu),M_{2}^\tau(tc)\leq a(t/\mu)x_2+b(t/\mu))}\\
&\to& P\Big(N([0,c]\times(x_2,\infty])=0\Big)\\
&& +P\Big(N([0,c]\times(x_1,\infty])=0,N([0,c]\times(x_2,x_1])=1\Big)\\
&=& e^{-c\mu(x_2,\infty]}+c\mu(x_2,x_1]e^{-c\mu(x_2,\infty]}\,.
\end{eqnarray*}

Another direct consequence of Theorem~\ref{thm_const} and \eqref{eq:SLLNCP} is the  functional limit theorem for corresponding extremal processes.
A partial results in this direction appears in \cite{PMM2009} where only the convergence of the one--dimensional distributions is proved.
For $t>0$, consider the random time changed extremal process
\begin{equation}\label{def_extremal_proc_finite_mean}
\xi_t(s)=\frac{M^{\tau}(ts)-b(t/\mu)}{a(t/\mu)}\,,
\quad s>0\,.
\end{equation}
Recall that an extremal process (see \cite{ResEVRVPP}) generated by 
an extreme value distribution function $G$ ($G$-extremal process, for short) is a continuous time stochastic process $\{\xi(s),s>0\}$ with finite dimensional distributions $G_{s_1,\ldots,s_k}$ 
satisfying
\[  G_{s_1,\ldots,s_k}(x_1,\ldots,x_k)=G^{s_1}(\wedge_{i=1}^k x_i)G^{s_2-s_1}(\wedge_{i=2}^k x_i)\cdots G^{s_k-s_{k-1}}(x_k)  \,, \]
 for all choices of $k\geq 1$, $0<s_1<\cdots < s_k$, $x_i\in\mathbb{R}$, $i=1,\ldots,k$.
In the proofs of Corollaries~\ref{corollary_finite_mean} and~\ref{corollary_infinite_mean} we will use the functional which maps point measures to the space of c\`adl\`ag functions and is given by
\begin{equation}\label{T_1 functional}
T_1 \left( \sum_{k}\delta_{(\tau_k,y_k)}\right)(t)=\bigvee_{\tau_k \leq t}y_k\,.
\end{equation}
Nice thing is that $T_1$ is a.s. continuous with respect to the distribution of $N$ in Theorem~\ref{thm_const} (see~\cite[p.214]{ResEVRVPP} for more details).
 
\begin{corollary}\label{corollary_finite_mean}
Assume that the assumptions of Theorem~\ref{thm_const} hold. Then, as $t\toi$, we have
\begin{equation*}
(\xi_t(s))_{s>0}\cid (\xi(s))_{s>0} \,,
\end{equation*}
in $D((0,\infty),\mathbb{R})$ with $J_1$ topology, where $(\xi_t(s))_{s>0}$ is defined in (\ref{def_extremal_proc_finite_mean}) and $(\xi(s))_{s>0}$ is $G$-extremal process. 
\end{corollary}

\begin{proof}
Observe that by Theorem 2.15 and Proposition 1.17 in \cite[Ch. VI]{JacodLTSP}, from (\ref{eq:SLLNCP}), we get  
$
\left(\mu\,\tau(tc)/t\right)_{c>0}\stackrel{a.s.}{\longrightarrow}(c)_{c>0}\,,
$
as $t\toi$, in the local uniform topology. Therefore, under the assumptions of Theorem~\ref{thm_const}, by Slutsky's theorem, we obtain the following joint convergence
\begin{equation*}
\left(N_t,\left(\frac{\mu\,\tau(tc)}{t}\right)_{c>0}\right)\cid \left(N,(c)_{c>0}\right)\,,
\end{equation*}
as $t \toi$.
This convergence takes place in the space $M_p([0,\infty)\times\mathbb{E})\times D(0,\infty)$ endowed with the product topology of vague and local uniform topology. From here, an application of functional $T_1$ given by~\eqref{T_1 functional} to the first coordinate and then Theorem 4 from~\cite{SilvestrovTeugels98} yields the statement.
\end{proof}

\section{Observations governed by an infinite mean renewal process independent of observations}\label{section_infinite_mean}

Throughout this section, we suppose that the renewal steps $Y$ have regularly varying distribution of  infinite mean with index $\alpha\in(0,1)$. In such a case, it is well known (see \cite{FellerVol2}) that there exists a strictly positive sequence $(d_n)$ such that
\[
d_n^{-1}(Y_1+\cdots +Y_n)\cid S_\alpha\,,
\]
where random variable $S_\alpha$ has the stable law with the index $\alpha$, scale parameter $\sigma=1$, skewness parameter $\beta=1$ and  shift parameter $\mu=0$. In particular, $S_\alpha$ is strictly positive a.s. The sequence $(d_n)$ can be chosen such that 
\begin{equation}\label{def_d(t)}
n(1-F_Y(d_n))\to 1\,,
\end{equation}
as $n\toi$, where $F_Y$ denotes cdf of $Y_1$. 
If we denote $d(t)=d_{\lfloor t\rfloor}$,  for $t \geq 0$,  with $d_0=0$ and 
$
 T(t)=\sum_{i=1}^{\lfloor t\rfloor}Y_i \,,
$
 then the function $d$ is regularly varying with index $1/\alpha$ and 
\begin{equation}\label{def_S_alpha}
 \left(\frac{T(tc)}{d(t)}\right)_{c\geq 0} \cid (S_\alpha(c))_{c\geq 0}\,,
\end{equation}
as $t\toi$, in a space of c\`adl\`ag functions $D[0,\infty)$ endowed with Skorohod $J_1$ topology (see \cite{Skorohod57} or \cite{ResHTP}). The limiting process $(S_\alpha(c))_{c \geq 0}$ is an  $\alpha$--stable process with  
strictly increasing sample paths.

Recall that for a function $z\in D([0,\infty),[0,\infty))$ the right continuous generalised inverse is defined by the relation
\begin{equation*}
z^{\leftarrow}(u)=\inf\{s\in[0,\infty):z(s)>u\} \,, \quad u \geq 0\,.
\end{equation*}
 According to \cite{Seneta76},  there exists a function $\widetilde{d}$ which is an asymptotic inverse of $d$, that is 
\begin{equation}\label{d_asympt_eq}
d(\widetilde{d}(t))\sim \widetilde{d}(d(t))\sim t\,,
\end{equation}
as $t\toi$. Moreover, $\widetilde{d}$ is regularly varying function with index $\alpha$. From (\ref{def_d(t)}) one can show
$
d(t)\sim\left(1/(1-F_Y)\right)^{\leftarrow}(t) \,, 
$
therefore
\begin{equation}\label{def_d_tilde}
\widetilde{d}(t)\sim\frac{1}{1-F_Y}(t) \,.
\end{equation}

Denote by
\[
W_\alpha(c)=\inf\{x:S_\alpha(x)>c\}=S_\alpha^{\leftarrow}(c)\,,
\quad c\geq 0\,,
\]
the first hitting-time process of the process $(S_\alpha(t))_{t\geq 0}$.
Now, Theorem 7.2 in \cite{Whitt80} (cf. \cite[p. 266]{ResHTP}) together with (\ref{def_S_alpha}) and (\ref{d_asympt_eq}) implies 
\begin{equation}\label{kvg_first_hitting_time}
\left(\frac{\tau(tc)}{\widetilde{d}(t)}\right)_{c\geq 0} \cid (W_\alpha(c))_{c\geq 0}\,,
\end{equation}
in $D([0,\infty),[0,\infty))$ endowed with $J_1$ topology, where $\tau(c)=T^{\leftarrow}(c)$ (see (\ref{def_renewal_process})). 
For an
$\alpha$--stable process $S_\alpha$ and a fixed $c\geq 0$, the hitting-time  $W_\alpha(c)$ has the Mittag-Leffler distribution (see e.g.~\cite{Anderson87})

If we assume independence between sequences $(X_n)$ and $(Y_n)$, 
the application of Lemma~\ref{basiclemma} becomes
relatively straightforward.
 In this subsection, the definition of point process $N_t$ induced by the sequence $(X_n)$ is changed, since different normalization of the process $(\tau(c))$ is needed in (\ref{kvg_first_hitting_time}). Namely, for $t> 0$ we define
\begin{equation*}
N_t=\sum_{i\geq 1}\delta_{\left(\frac{i}{\widetilde{d}(t)},\widetilde X_{t,i}\right)} \,,
\end{equation*}
 where $\widetilde X_{t,i}$ is defined by
\begin{equation}\label{def_tilde_X_ti}
\widetilde X_{t,i}=\frac{X_i-\widetilde{b}(t)}{\widetilde{a}(t)}\,,
\end{equation}
with $\widetilde{a}(t):=a(\widetilde{d}(t))$, $\widetilde{b}(t):=b(\widetilde{d}(t))$ and $a(t)$, $b(t)$,  $\widetilde{d}(t)$ defined in (\ref{MDA_condition}) and $(\ref{def_d_tilde})$.
The following theorem describes the asymptotic behaviour of all upper order statistics in the sequence of observations $(X_n)$
 separated by regularly varying waiting times of infinite mean and independent of observations $(X_n)$.
\begin{thm}\label{thm_inf_mean_independence}
Suppose that $(X_n)$ and $(Y_n)$ are independent iid sequences such that $X_1\in \mathrm{MDA}(G)$, $Y_1\sim \mathit{RegVar}(\alpha)$ with $0<\alpha<1$. Then, for every $c \geq 0$
\begin{equation}\label{eq:N_t_cvg_inf_mean_independence}
N_t\Big|_{\left[0,\frac{\tau(tc)}{\widetilde{d}(t)}\right]\times\mathbb{E}} \cid  N\Big|_{[0,W_\alpha(c)]\times\mathbb{E}} \,,
\end{equation} 
as $t\toi$, where $N$ is $\mathrm{PRM}(\lambda\times\mu_G)$ independent of the process $(W_\alpha(c))_{c \geq 0}$ distributed as in (\ref{kvg_first_hitting_time}).
\end{thm}
\begin{proof}
Since $X_1\in \mathrm{MDA}(G)$ and $\widetilde{d}(t)\nearrow\infty$ as $t\toi$, (\ref{MDA_condition}) is equivalent to $N_t\cid N$ as $t\toi$, where $N$ is $\mathrm{PRM}(\lambda\times \mu_G)$ (see Proposition 3.21 in \cite{ResEVRVPP}). Using (\ref{kvg_first_hitting_time}) and the assumption of independence between processes  $(N_t)$ and $(\tau(t))_{t\geq 0}$, we obtain
\begin{equation}\label{thm_indep_joint_cvg_process_ver}
 \left(N_t,\left(\frac{\tau(tc)}{\widetilde{d}(t)}\right)_{c \geq 0}\right)\cid (N,(W_\alpha(c))_{c \geq 0})\,,
\end{equation}
as $t\toi$. Hence, for fixed $c \geq 0$, we have
\[ \left(N_t,\frac{\tau(tc)}{\widetilde{d}(t)}\right)\cid (N,W_\alpha(c))\,, \]
and, application of Lemma~\ref{basiclemma} yields the desired result.
\end{proof}

Direct consequence of Theorem~\ref{thm_inf_mean_independence} is the asymptotic behaviour of the $k$-th upper order statistics in a sample indexed by the renewal process $(\tau(t))$.
Recall that $M_k^{\tau}(t),\  t\geq 0$ represent the $k$-th upper order statistics in a sample of observations $\{X_1\ldots,X_{\tau(t)}\}$, where $(\tau(t))_{t \geq 0}$ is the renewal process. Since
\[ \left\{\frac{M_k^\tau(t)-\widetilde{b}(t)}{\widetilde{a}(t)}\leq x\right\}=\left\{N_t\left(\Big[0,\frac{\tau(t)}{\widetilde{d}(t)}\Big]\times(x,\infty]\right)\leq k-1\right\}\,, \]
from \eqref{eq:N_t_cvg_inf_mean_independence}
we obtain
\begin{eqnarray*}
P\left(M_k^\tau(t)\leq \widetilde{a}(t)x+\widetilde{b}(t)\right)&\to& P\Bigg(N\Big([0,W_\alpha(1)]\times (x,\infty]\Big)\leq k-1\Bigg)\\
&=&E\left(\frac{\Gamma(k,W_\alpha(1)\mu_G(x,\infty])}{\Gamma(k)}\right)\,, 
\end{eqnarray*}
as $t\toi$, where $\Gamma(k,x)$ is an incomplete gamma function (see \cite{Stegun}). For $k=1$, i.e. for the partial maxima of the first $\tau(t)$ observations, the result first appears in \cite{Berman62} (cf. Theorems 2.1, 2.2 and 2.3 in~\cite{Berman62}).

Using the approach presented in this section we can easily recover the functional limit theorem for a random time changed extremal processes derived in~\cite{MStoev2009}.
For a fixed $t > 0$, we consider the following extremal process
\begin{equation}\label{def_extremal_proc_infinite_mean_indep}
\widetilde{\xi}_t(s)=\frac{M^\tau(ts)-\widetilde{b}(t)}{\widetilde{a}(t)}\,,\quad s>0\,,
\end{equation}
where $\widetilde{a}(t)$ and $\widetilde{b}(t)$ are defined in \eqref{def_tilde_X_ti}. 
Clearly, the process $(W_\alpha(c))_{c>0}$ has nondecreasing sample paths. Thus,
 for  a $G$-extremal process ${(\xi(s))_{s>0}}$ independent of the process $(W_\alpha(c))_{c>0}$, 
a subordinated process
\begin{equation}\label{def_subordinated_G_extrem_proc}
(\widetilde{\xi}_W(s))_{s>0}=(\xi(W_\alpha(s)))_{s>0}\,,
\end{equation}
is well defined and nondecreasing as well.
Once again, an application of $T_1$ functional (see~\eqref{T_1 functional}) to \eqref{thm_indep_joint_cvg_process_ver} and Theorem 4 from~\cite{SilvestrovTeugels98} yields the following result.
\begin{corollary}\label{corollary_infinite_mean}
Under the assumptions in this section, as $t\toi$, we have
\begin{equation*}
(\widetilde{\xi}_t(s))_{s>0}\cid (\widetilde{\xi}_W(s))_{s>0}\,,
\end{equation*}
in $D((0,\infty),\mathbb{R})$ with $J_1$ topology, where $\widetilde{\xi}_t(\cdot)$ and $\widetilde{\xi}_W(\cdot)$ are defined in $(\ref{def_extremal_proc_infinite_mean_indep})$ and $(\ref{def_subordinated_G_extrem_proc})$, respectively.
\end{corollary}



\section*{Acknowledgements}
We  thank P.Goldstein for careful reading of the manuscript and
many thoughtful suggestions. 
This work has been supported in part by Croatian Science Foundation under projects 3526 and 1356.

  \bibliographystyle{plain} 
  \bibliography{mybib_initials}

\end{document}